%% file: Shift_Fq_Arxiv.tex
\documentclass[review,onefignum,onetabnum]{siamart171218}

\input{ex_shared}

\begin{document}

\maketitle

\begin{abstract}
Let $\mathbb{F}_q$ be a finite field of order $q$, where $q$ is a power of a prime. For a set $A \subset \mathbb{F}_q$, under certain structural restrictions, we prove a new explicit lower bound on the size of the product set $A(A + 1)$. Our result improves on the previous best known bound due to Zhelezov and holds under more relaxed restrictions.
\end{abstract}



\section{Introduction}
Let $p$ denote a prime, $\mathbb{F}_q$ the finite field consisting of $q = p^m$ elements and $\mathbb{F}^{*}_q = \mathbb{F}_q\backslash\{0\}$. For sets $A, B \subset \mathbb{F}_q$, we define the sum set $A + B = \{ a + b : a\in A, b\in B\}$ and the product set $AB = \{ ab : a\in A, b\in B\}$.
Similarly, we define the difference set $A-B$ and the ratio set $A/B$.

The sum-product phenomenon in finite fields is the assertion that for $A \subset \mathbb{F}_q$, the sets $A+A$ and $AA$ cannot both simultaneously be small unless $A$ closely correlates with a coset of a subfield. A result in this direction is due to Li and Roche-Newton~\cite{RoLi}, who showed that if $|A\cap cG| \leq |G|^{1/2}$ for all subfields $G$ and elements $c$ in $\mathbb{F}_q$, then
 \begin{displaymath}
\max\{|A+A|, |AA|\} \gg (\log |A|)^{-5/11}|A|^{1+1/11}.
 \end{displaymath}

In the same spirit and under a similar structural assumption on the set $A$, one expects that, for all $\alpha \in \mathbb{F}_q^*$, either of the product sets $AA$ or $(A+\alpha)(A+\alpha)$ must be significantly larger than $A$. Zhelezov~\cite{Zhel} proved the estimate
\begin{equation}
\label{eqn:ZhelBound}
\max \{|AB|, |(A + 1)C|\} \gtrsim |A|^{1+1/559},
\end{equation}
for sets $A, B, C \subset \mathbb{F}_q$, under the condition that
\begin{equation}
\label{eqn:ZhelCon}
|AB \cap cG| \leq |G|^{1/2}
\end{equation}
for all subfields $G$ of $\mathbb{F}_q$ and elements $c\in \mathbb{F}_q$. Then, taking $B = A$ and $C = A+1$, under restriction \eqref{eqn:ZhelCon}, we have
\begin{equation}
\label{eqn:ZhelB1}
\max \{|AA|, |(A + 1)(A+1)|\} \gtrsim |A|^{1+1/559}.
\end{equation}

For sets $B_1, B_2 ,X \subset \mathbb{F}_q^*$, we recall Pl\"{u}nnecke's inequality (see Lemma~\ref{lem:PlRu})
 \begin{displaymath}
|B_1B_2| \leq\frac{|B_1X||B_2X|}{|X|}.
 \end{displaymath}
From this we can deduce that
 \begin{displaymath}
|A(A+1)|^{2} \geq |A| \cdot \max\{|AA|, |(A+1)(A+1)|\}.
 \end{displaymath}
Hence, by \eqref{eqn:ZhelB1}, we have the estimate
\begin{equation}
\label{eqn:ZhelB2}
|A(A+1)| \gtrsim |A|^{1 + \delta}
\end{equation}
with $\delta = 1/1118$, which holds under restriction~\eqref{eqn:ZhelCon} with $B=A$. Alternatively, by \eqref{eqn:ZhelBound}, with $B= A+1$ and $C = A$, the estimate \eqref{eqn:ZhelB2} holds with $\delta = 1/559$.

For large sets, $A\subset \mathbb{F}_q$ with $|A| \geq q^{1/2}$, Garaev and Shen~\cite{GarShe} proved the bound
\begin{equation}
\label{eqn:GSB}
|A(A+1)| \gg \min\{q^{1/2}|A|^{1/2}, |A|^{2}/q^{1/2}\}.
\end{equation}
Furthermore, it was demonstrated in \cite{GarShe} that in the range $|A| > q^{2/3}$, the bound \eqref{eqn:GSB} is optimal up to the implied constant.

In the realm of small sets $A\subset \mathbb{F}_q$, with $|A| \ll p^{5/8}$, Stevens and de~Zeeuw~\cite{StevZeeuw} obtained
 \begin{displaymath}
|A(A+1)| \gg |A|^{1+1/5}.
 \end{displaymath}
This result is based on a bound on incidences between points and lines in Cartesian products, proved in the same paper, which itself relies on a bound on incidences between points and planes due to Rudnev~\cite{Rud2}. We point out that the main result of \cite{Rud2} has led to many quantitatively strong sum-product type estimates, however these estimates are restricted to sets of size smaller than $p$.

Our main result, stated below, relies on a somewhat more primitive approach towards the sum-product problem in finite fields, often referred to as the additive pivot technique. Specifically, we adopt our main tools and ideas from \cite{JRN} and \cite{RoLi}.
\begin{theorem}
\label{thm:ShiftExpander}
Let $A \subseteq \mathbb{F}_q$. Suppose that
\begin{equation}
\label{eqn:ShiftECon}
 |A \cap cG| \ll \max\{|G|^{1/2}, |A|^{25/26}\} 
\end{equation}
for all proper subfields $G$ of $\mathbb{F}_q$ and elements $c\in \mathbb{F}_q$. Then for all $\alpha \in \mathbb{F}_q^{*}$, we have
 \begin{displaymath}
|A(A + \alpha)| \gtrsim\min\big\{|A|^{1+1/52}, q^{1/48}|A|^{1-1/48}\big\}.
 \end{displaymath}
\end{theorem}
Theorem~\ref{thm:ShiftExpander} provides a quantitative improvement over the relevant estimates implied by \eqref{eqn:ZhelBound} and holds under a more relaxed condition than those given by \eqref{eqn:ZhelCon}. It also improves on \eqref{eqn:GSB} in the range $q^{1/2} \leq |A| \lesssim q^{1/2 + 1/102}$.

Given a set $A \subset \mathbb{F}_q$, we define the additive energy of $A$ as the quantity
 \begin{displaymath}
E_{+}(A) = |\{ (a_1, a_2, a_3, a_4) \in A^{4} : a_1 + a_2 = a_3 + a_4\}|.
 \end{displaymath}
As an application of Theorem~\ref{thm:ShiftExpander}, we record a bound on the additive energy of subsets of $\mathbb{F}_q$.
\begin{corollary}
\label{cor:AdEnergy}
Let $A \subseteq \mathbb{F}_q$. Suppose that
\begin{equation}
\label{eqn:AdEnergyCon}
 |A \cap cG| \ll \max\{|G|^{1/2}, |AA|^{50/53}\} 
\end{equation}
for all proper subfields $G$ of $\mathbb{F}_q$ and elements $c\in \mathbb{F}_q$. Then for any $\alpha \in \mathbb{F}_q^{*}$, we have
\begin{equation}
\label{eqn:ShiftInter}
|A\cap (A-\alpha)| \lesssim |AA|^{1-1/53} + q^{-1/47}|AA|^{1+1/47}.
\end{equation}
Consequently, under restriction \eqref{eqn:AdEnergyCon}, we have
 \begin{displaymath}
E_{+}(A) \lesssim |A|^{2}\big(|AA|^{1-1/53} + q^{-1/47}|AA|^{1+1/47}\big).
 \end{displaymath}
\end{corollary}

\subsection*{Asymptotic notation}
We use standard asymptotic notation. In particular, for positive real numbers $X$ and $Y$, we use $X = O(Y)$ or $ X \ll Y$ to denote the existence of an absolute constant $c >0$ such that $X \leq cY$. If $X \ll Y$ and $Y \ll X$, we write $X  = \Theta(Y)$ or $X \approx Y$. We also use $X \lesssim Y$ to denote the existence of an absolute constant $c >0$, such that $X \ll (\log Y)^{c} Y$.

\section{Preparations}
For $X \subset \mathbb{F}_q$, let $R(X)$ denote the quotient set of $X$, defined by
 \begin{displaymath}
R(X) =\bigg\{\frac{x_1 - x_2}{x_3 - x_4} : x_1, x_2, x_3, x_4 \in X, x_3 \neq x_4\bigg\}.
 \end{displaymath}
We present a basic extension of~\cite[Lemma 2.50]{TaoVu}.
\begin{lemma}
 \label{lem:RBcard}
 Let $X \subset \mathbb{F}_q$ and $r \in \mathbb{F}_q^{*}$. If $r\not\in R(X)$, for any nonempty subsets $X_1, X_2 \subseteq X$, we have
 \begin{displaymath}
 |X_1||X_2| = |X_1 - rX_2|.
 \end{displaymath}
\end{lemma}

\begin{proof}
Consider the mapping $\phi: X_1 \times X_2 \rightarrow X_1 - rX_2$ defined by $\phi(x_1, x_2) = x_1 - rx_2$. Suppose that $(x_1, x_2), (y_1, y_2) \in X_1\times X_2$ are distinct pairs satisfying $x_1 - rx_2 = y_1 - ry_2$. Then we get 
 \begin{displaymath}
r = \frac{x_1 - y_1}{x_2 - y_2},
 \end{displaymath}
which contradicts the assumption that $r\not\in R(X)$.  We deduce that $\phi$ is injective, which in turn implies the required result.
\end{proof}

The next lemma, which appeared in \cite[Corollary~2.51]{TaoVu}, is a simple corollary of Lemma~\ref{lem:RBcard}.
\begin{lemma}
\label{lem:RBFq}
Let $X \subset \mathbb{F}_q$ with $|X| > q^{1/2}$, then $R(X) = \mathbb{F}_q$.
\end{lemma}

We have extracted Lemma~\ref{lem:QuotientSetSubfield}, stated below, from the proof of the main result in~\cite{RoLi}.
\begin{lemma}
\label{lem:QuotientSetSubfield}
Let $X \subset \mathbb{F}_q$ be such that 
 \begin{displaymath}
1+R(X) \subseteq R(X) \quad \text{and} \quad X\cdot R(X) \subseteq R(X).
 \end{displaymath}
Then $R(X)$ is the subfield of $\mathbb{F}_q$ generated by $X$.
\end{lemma}

The next result has been stated and proved in the proof of \cite[Theorem~1]{Rud}.
\begin{lemma}
\label{lem:pivotting}
Let $X \subset \mathbb{F}_q$ with $|R(X)| \gg |X|^2$. Then there exists $r\in R(X)$ such that for any subset $X^{'} \subset X$ with $|X^{'}| \approx |X|$, we have 
 \begin{displaymath}
|X^{'} + rX^{'}| \gg |X|^2.
 \end{displaymath}
\end{lemma}

The following lemma enables us to extend our main result to sets which are larger than $q^{1/2}$. See \cite[Lemma~3]{BouGlib} for a proof.
\begin{lemma}
\label{lem:BouGlibPivot}
Let $X_1, X_2 \subset \mathbb{F}_q$. There exists an element $\xi \in \mathbb{F}_q^{*}$ such that
 \begin{displaymath}
|X_1 + \xi X_2| \geq \frac{|X_1||X_2|(q-1)}{|X_1||X_2| + (q - 1)}.
 \end{displaymath}
\end{lemma}

Next, we recall Ruzsa's triangle inequality. See \cite[Lemma~2.6]{TaoVu} for a proof.
\begin{lemma}
\label{lem:RuszaTriangle}
Let $X, B_1, B_2$ be nonempty subsets of an abelian group. We have 
 \begin{displaymath}
 |B_1 - B_2| \le \frac{|X + B_1||X + B_2|}{|X|}.
 \end{displaymath}
\end{lemma}

In particular, for $A\subset \mathbb{F}_q^*$, by a multiplicative application of Lemma~\ref{lem:RuszaTriangle}, we have the useful inequality 
\begin{equation}
\label{eqn:RatioToShift}
|A/A| \leq \frac{|A(A+1)|^{2}}{|A|}.
\end{equation}

In the next two lemmas we state variants of the Pl\"{u}nnecke-Ruzsa inequality, which can also be found in \cite{KatzShe}.
\begin{lemma}
\label{lem:PlRu}
Let $X, B_1, \dots, B_k$ be nonempty subsets of an abelian group. Then 
 \begin{displaymath}
|B_1 + \cdots + B_k| \leq \frac{|X + B_1| \cdots |X + B_k|}{|X|^{k-1}}.
 \end{displaymath}
\end{lemma}

\begin{lemma}
\label{lem:PlRuRefined}
Let $X, B_1, \dots, B_k$ be nonempty subsets of an abelian group. For any $0 < \epsilon < 1$, there exists a subset $X^{'} \subseteq X$, with $|X^{'}| \geq (1-\epsilon)|X|$ such that 
 \begin{displaymath}
|X^{'} +B_1 + \cdots + B_k| \ll_{\epsilon} \frac{|X + B_1| \cdots |X + B_k|}{|X|^{k-1}}. 
\end{displaymath}
\end{lemma}

The following two lemmas are due to Jones and Roche-Newton~\cite{JRN}.
\begin{lemma}
\label{lem:CoveringByShifts}
Let $Z \subseteq \mathbb{F}_q^*$. Suppose that $X, Y \subseteq  xZ + y$ for some $x \in \mathbb{F}_q^{*}$ and $y \in \mathbb{F}_q$. Fix $0 < \epsilon < 1/16$. Then,  $(1-\epsilon)|X|$ elements of $X$ can be covered by 
 \begin{displaymath}
 O_{\epsilon} \bigg(\frac{|Z(Z+1)|^2 |Z/Z|}{|X||Y|^2}\bigg)
 \end{displaymath}
translates of $Y$. Similarly, $(1-\epsilon)|X|$ elements of $X$ can be covered by this many translates of $-Y$.
\end{lemma}

\begin{lemma}
\label{lem:BasicShiftBound}
Let $A \subseteq \mathbb{F}_q^*$. There exists a subset $A^{'} \subseteq A$ with $|A^{'}| \approx |A|$ such that 
 \begin{displaymath}
|A^{'} - A^{'}| \ll \frac{|A(A+1)|^4|A/A|^2}{|A|^5}.
 \end{displaymath}
\end{lemma}

Next, we record a popularity pigeonholing argument. A proof is provided in \cite[Lemma~9]{Jones}.
\begin{lemma}
\label{lem:popularity}
Let $X$ be a finite set and let $f$ be a function such that $f(x) > 0$ for all $x \in X$. Suppose that 
 \begin{displaymath}
\sum_{x\in X} f(x) \geq K.
 \end{displaymath}
Let $Y = \{ x \in X : f(x) \geq K/2|X|\}.$ Then 
 \begin{displaymath}
\sum_{y\in Y} f(y) \geq \frac{K}{2}.
 \end{displaymath}
Additionally, if $f(x) \leq M$ for all $x\in X$, then $|Y| \geq K/(2M)$.
\end{lemma}

For sets $X, Y \subseteq \mathbb{F}_q$, we define the multiplicative energy between $X$ and $Y$ as the quantity
 \begin{displaymath}
E_{\times}(X, Y) = |\{ (x_1, x_2, y_1, y_2) \in X^{2}\times Y^{2} : x_1 y_1 = x_2 y_2\}|
 \end{displaymath}
and write simply $E_{\times}(X)$ instead of $E_{\times}(X, X)$. For $\xi \in Y/X$, let 
 \begin{displaymath}
r_{Y:X}(\xi) = |\{ (x, y) \in X\times Y : y/ x = \xi\}|.
 \end{displaymath}
Then, we have the identities
\begin{equation}
\label{eqn:r1stM}
\sum_{\xi \in Y/X} r_{Y:X}(\xi) = |X||Y|,
\end{equation}
\begin{equation}
\label{eqn:r2ndM}
\sum_{\xi \in Y/X} r_{Y:X}^{2}(\xi) = E_{\times}(X, Y).
\end{equation}
By a simple application of the Cauchy-Schwarz inequality we have
\begin{equation}
\label{eqn:EnergyCS}
E_{\times}(X, Y)|XY| \geq |X|^{2}|Y|^{2}.
\end{equation}

The remaining two lemmas together form the basis for the proof of Theorem~\ref{thm:ShiftExpander}. Lemma~\ref{lem:Rudnev} is a slight generalisation of \cite[Lemma~3]{Rud}.
\begin{lemma} 
\label{lem:DyadicEnergy}
Let $X, Y \subset \mathbb{F}_q$, with $|Y|\leq |X|$. There exists a set $D \subseteq Y/X$ and an integer $N\leq |Y|$ such that $E_{\times}(X, Y) \ll (\log|X|)|D|N^{2}$ and $|D|N < |X||Y|$. Also, for $\xi \in D$ we have $r_{Y:X}(\xi) \approx N$. Namely, the set of points
 \begin{displaymath}
P = \{(x, y) \in X \times Y: y/x \in D\}
 \end{displaymath}
is supported on $|D|$ lines through the origin, with each line containing $\Theta(N)$ points of $P.$ 
\end{lemma}
\begin{proof}
For $j\geq 0$, let $L_j = \{\xi \in Y/X : 2^{j} \leq r_{Y:X}(\xi) < 2^{j+1}\}$. Then, by \eqref{eqn:r2ndM}, we have
 \begin{displaymath}
 \sum_{j=0}^{\lfloor\log_2|X|\rfloor}\sum_{\xi\in L_j}r_{Y:X}^{2}(\xi) = E_{\times}(X, Y).
 \end{displaymath}
By the pigeonhole principle there exists some $N\geq 1$ such that, letting $D = \{\xi \in Y/X : N \leq r_{Y:X}(\xi) < 2N\}$, we have 
 \begin{displaymath}
 \frac{E_{\times}(X, Y)}{\log |X|} \ll \sum_{\xi\in D}r_{Y:X}^{2}(\xi) \ll |D|N^{2}.
 \end{displaymath}
Furthermore, by \eqref{eqn:r1stM}, we have
 \begin{displaymath}
|D|N < \sum_{\xi\in D}r_{Y:X}(\xi) \leq |X||Y|.
 \end{displaymath}
\end{proof}

\begin{lemma}
\label{lem:Rudnev}
Let $X, Y \subset \mathbb{F}_q$. Suppose $P \subset X\times Y$ is a set of points supported on $L$ lines through the origin, with each line containing $\Theta(N)$ points of $P$, so that $|P| \approx LN$. For $x_{*} \in X$ and $y_{*} \in Y$, we write $Y_{x_{*}} = \{y\in Y: (x_{*}, y) \in P\}$ and $X_{y_{*}} = \{x \in X: (x, y_{*}) \in P\}.$ There exists a popular abscissa $x_0$ and a popular ordinate $y_0$, so that 
 \begin{displaymath}
|Y_{x_0}| \gg \frac{LN}{|X|}, \quad |X_{y_0}| \gg \frac{LN}{|Y|}.
 \end{displaymath}
For $\xi \in \mathbb{F}_q$, we write $P_{\xi} = \{x: (x, \xi x) \in P\}$. There exists a subset $\widetilde{Y}_{x_0} \subseteq Y_{x_0}$ with
\begin{equation}
|\widetilde{Y}_{x_0}| \gg \frac{L^{2}N^{2}}{|X|^{2}|Y|},
\end{equation}
such that for every $z\in \widetilde{Y}_{x_0}$, we have
\begin{equation}
|P_{z/x_{0}} \cap X_{y_0}| \gg \frac{L^{2}N^{3}}{|X|^{2}|Y|^{2}}.
\end{equation}
\end{lemma}

\begin{proof}
Observing that
 \begin{displaymath}
\sum_{y\in Y} |X_y| = |P| \approx LN,
 \end{displaymath}
by Lemma~\ref{lem:popularity}, there exists a subset $Y^{'} \subseteq Y$ such that, for all $y\in Y^{'}$, we have $|X_y| \gg LN/|Y|$. Let $P^{'}  = \{(x, y) \in P: y \in Y^{'}\}$ so that $|P^{'}| \gg LN$. Then 
 \begin{displaymath}
\sum_{x\in X} |Y_x\cap Y^{'}| = \sum_{y\in Y^{'}}|X_y| = |P^{'}|  \gg LN.
 \end{displaymath}
By Lemma~\ref{lem:popularity}, there exists a subset $X^{'} \subseteq X$ such that for all $x\in X^{'}$ we have 
\begin{equation}
\label{eqn:YxPopBound}
 |Y_x\cap Y^{'}|  \gg \frac{LN}{|X|}.
\end{equation}
Letting $P^{''} = \{(x, y) \in P^{'}: x \in X^{'}\},$ we have $|P^{''}| \gg LN.$

Let $D = \{y/x : (x, y) \in P^{''}\}$ and let $D^{'} \subseteq D$ denote the set of elements $\xi$ such that the lines $l_{\xi} $, determined by $\xi$, each contain $\Omega(N)$ points of $P^{''}$. It follows by Lemma~\ref{lem:popularity} that $|D^{'}| \gg L$. Now, we proceed to establish a lower bound on the sum
\begin{equation}
\label{eqn:sigma}
\Sigma = \sum_{(x, y) \in X^{'}\times Y^{'} }\sum_{z\in Y_{x}}|P_{z/x}\cap X_y|.
\end{equation}
We write $z\sim x$, if $(x, z)$ is a point of $P$. Then
\begin{align*}
\label{eqn:Sigma2}
\Sigma &\gg\sum_{\substack{(x, y) \in X^{'}\times Y^{'} \\ z:z\sim x}} |P_{z/x}\cap X_y|  \\ &\gg N\sum_{\xi \in D^{'}}\sum_{y\in Y^{'}}|P^{''}_{\xi}\cap X_y|.
\end{align*}
For a fixed $\xi \in D^{'}$, the inner sum may be bounded by the observation that 
 \begin{displaymath}
\sum_{y\in Y^{'}}|P^{''}_{\xi}\cap X_y| = \sum_{x\in P^{''}_{\xi}}|Y_x\cap Y^{'}|.
 \end{displaymath}
Recall that $|D^{'}|\gg L$ and that for $\xi \in D^{'}$, we have $|P^{''}_{\xi}|\gg N$. Then, by \eqref{eqn:YxPopBound}, we have
 \begin{displaymath}
\Sigma \gg N \cdot L \cdot N \cdot \frac{LN}{|X|}.
 \end{displaymath}
By the pigeonhole principle, applied to \eqref{eqn:sigma}, there exist $(x_0, y_0) \in X^{'} \times Y^{'}$ such that 
 \begin{displaymath}
\sum_{z\in Y_{x_0}}|P_{z/x_{0}}\cap X_{y_0}| \gg \frac{L^{2}N^{3}}{|X|^{2}|Y|}.
 \end{displaymath}
By our assumption, that every line through the origin contains $O(N)$ points of $P$, it follows that for all $z\in Y$, we have $|P_{z/x_{0}}| \ll N$. Then, letting $\widetilde{Y}_{x_0} \subseteq Y_{x_0}$ to denote the set of $z\in Y_{x_0}$ with the property that
 \begin{displaymath}
|P_{z/x_{0}}\cap X_{y_0}| \gg \frac{L^{2}N^{3}}{|X|^{2}|Y|^{2}},
 \end{displaymath}
by Lemma~\ref{lem:popularity}, we have 
 \begin{displaymath}
|\widetilde{Y}_{x_0}| \gg \frac{L^{2}N^{2}}{|X|^{2}|Y|}.
 \end{displaymath}
\end{proof}

\section{Proof of Theorem~\ref{thm:ShiftExpander}}
It suffices to prove the required result for $\alpha =1$. Then the general statement immediately follows since under condition~\eqref{eqn:ShiftECon} the set $A$ can be replaced by any of its dilates $cA$, for $c\in \mathbb{F}_q^*$. Without loss of generality assume $0 \not\in A$. By Lemma~\ref{lem:BasicShiftBound}, combined with \eqref{eqn:RatioToShift}, there exists a subset $A^{'} \subseteq A$, with $|A^{'}| \approx |A|$, such that 
 \begin{displaymath}
|A^{'} - A^{'}| \ll \frac{|A(A+1)|^8}{|A|^7}.
 \end{displaymath}
By Lemma~\ref{lem:PlRuRefined} there exists a further subset $A^{''} \subseteq A^{'}$, with $|A^{''}| \approx |A^{'}|$, such that
 \begin{displaymath}
|A^{''} -A^{''} -A^{''} -A^{''}| \ll \frac{|A^{'} - A^{'}|^3}{|A|^2}.
 \end{displaymath}
Since $|A^{''}| \approx |A|$, we reset the notation $A^{''}$ back to $A$ and henceforth assume the inequalities
\begin{equation}
\label{differencebound}
|A -A| \ll \frac{|A(A+1)|^8}{|A|^7},
\end{equation}
\begin{equation}
\label{iteratednound}
|A -A -A -A| \ll \frac{|A(A+1)|^{24}}{|A|^{23}}.
\end{equation}
We apply Lemma~\ref{lem:DyadicEnergy} to identify a set $D \subseteq A/(A+1)$ and an integer $N\geq 1$ such that for $\xi \in D$ we have $r_{A:(A+1)}(\xi) \approx N$. Additionally, letting $L = |D|$, in view of \eqref{eqn:EnergyCS}, we have 
\begin{equation}
\label{eqn:SEMECS}
M:= LN^{2} \gg \frac{E_{\times}(A+1, A)}{\log|A|} \geq \frac{|A|^{4}}{|A(A+1)|\log|A|}.
\end{equation}
We define $P \subseteq (A+1) \times A$ by
 \begin{displaymath}
P = \{(x, y) \in (A+1) \times A: y/x \in D\}.
 \end{displaymath}
Then $|P| \approx LN$.
Now, since $LN < |A|^{2}$ and $N< |A|$, we get
\begin{equation}
\label{eqn:NLB}
N, L > \frac{M}{|A|^{2}}.
\end{equation}
For $\xi \in D$, we define the projection onto the $x$-axis of the line with slope $\xi$ as
 \begin{displaymath}
P_{\xi} = \{x:(x, \xi x) \in P\} \subset A+ 1.
 \end{displaymath}
Similarly for $\lambda \in D^{-1}$ let
 \begin{displaymath}
Q_{\lambda} = \{y:(\lambda y,  y) \in P\} \subset A.
 \end{displaymath}
Then for $\xi \in D$ and $\lambda \in D^{-1}$, we have
\begin{equation}
\label{eqn:ProjectionsSizes}
|P_{\xi}|, |Q_{\lambda}| \approx N, \quad  \xi P_{\xi} \subset A\quad \text{and}\quad \lambda Q_{\lambda} \subset A + 1.
\end{equation}
By Lemma~\ref{lem:Rudnev}, with $X= A+1$ and $Y = A$, there exists a pair of elements $(x_0, y_0) \in (A+1)\times A$ such that the sets $A_{x_0} \subseteq A$ and $B_{y_0} \subseteq A+1$ satisfy 
\begin{equation}
\label{eqn:Ax0By0Sizes}
|A_{x_0}|, |B_{y_0}| \gg \frac{LN}{|A|}, \quad x^{-1}_{0}A_{x_0} \subset D \quad \text{and} \quad y^{-1}_{0}B_{y_0} \subset D^{-1}.
\end{equation}
Moreover, there exists a further subset $\tilde{A}_{x_0} \subseteq A_{x_0}$, with 
\begin{equation}
\label{A1Size}
|\tilde{A}_{x_0}| \gg \frac{LM}{|A|^{3}},
\end{equation}
such that for all $z \in \tilde{A}_{x_0}$, letting $S_z = P_{z/x_{0}} \cap B_{y_0}$, we have 
\begin{equation}
\label{eqn:SzSize}
|S_z| \gg \frac{LMN}{|A|^{4}}.
\end{equation}

We require the following corollary of Lemma~\ref{lem:CoveringByShifts} throughout the remainder of the proof.
\begin{claim}
\label{coveringapp}
For $n\leq 4$ let $a_1, \dots, a_n$ denote arbitrary elements of $\tilde{A}_{x_0}$. Given any set $C\subset A+1$, there exists a subset $C^{'} \subset C$, with $|C^{'}|\approx |C|$, such that the sets $a_{i}C^{'}$ can each be covered by
\begin{equation}
\label{eqn:1stCApp}
O\bigg(\frac{|A(A+1)|^4}{|C||A|N^{2}}\bigg)
\end{equation}
translates of $\pm x_0 A$. 

Suppose $b_1, \dots, b_4 \in B_{y_0}$. Let
\begin{equation}
\label{eqn:2ndCApp}
\Gamma:= \frac{|A|^{2}|A(A+1)|^4}{M^2}.
\end{equation}
There exists a subset $A^{'} \subseteq \tilde{A}_{x_0}$, with $|A^{'}| \approx |\tilde{A}_{x_0}|$, such that for $1\leq i\leq 4$ the sets $b_{i}A^{'}$ can each be covered by $O(\Gamma)$ translates of $\pm y_0 A$. 
\end{claim}
\begin{proof}
We apply Lemma~\ref{lem:CoveringByShifts}, with $X=a_{i}C$, $Y= a_{i}P_{a_{i}/x_0}$, $Z= A$, $x = a_i,$ $y=a_i$ and $0 < \epsilon < 1/16$. Then there exist sets $C_{a_i} \subseteq C$ with $|C_{a_i}| \geq (1-\epsilon)|C|$ such that each of $a_{i}C_{a_i}$ can be covered by 
 \begin{displaymath}
 O_{\epsilon} \bigg(\frac{|A(A+1)|^2 |A/A|}{|C||a_{i}P_{a_{i}/x_0}|^2}\bigg) 
 \end{displaymath}
translates of $a_{i}P_{a_{i}/x_0} \subset x_0 A$ and by at most as many translates of $-x_0 A$. 
Let $C^{'}= C_{a_1} \cap\dots \cap C_{a_n}$, so that $|C^{'}| \geq (1 - n\epsilon)|C| \geq (3/4) |C|$. Then, by \eqref{eqn:RatioToShift} and \eqref{eqn:ProjectionsSizes}, it follows that \eqref{eqn:1stCApp} denotes the number of translates of $\pm x_0 A$ required to cover the sets $a_{i}C^{'}$ for $1 \leq i \leq n$.

Next, we apply Lemma~\ref{lem:CoveringByShifts}, with $X=b_{i}\tilde{A}_{x_0}$, $Y= b_{i}Q_{b_{i}/y_{0}}$, $Z= A$, $x = b_{i}$ and $y= 0$. Recalling \eqref{eqn:ProjectionsSizes}, \eqref{eqn:Ax0By0Sizes}, \eqref{A1Size} and proceeding similarly as above, we can identify a subset $A^{'} \subseteq \tilde{A}_{x_0}$, with $|A^{'}| \approx |\tilde{A}_{x_0}|$, such that the sets $b_{i}A^{'}$ are each fully contained in $O(\Gamma)$ translates of $\pm y_0A$. 
\end{proof}

We proceed to split the proof into four cases based on the nature of the quotient set $R(\tilde{A}_{x_0})$.

\subsection*{ Case 1:\nopunct} $R(\tilde{A}_{x_0} ) \neq R(B_{y_{0}})$.
\subsection*{ Case 1.1:\nopunct}
There exist elements $a, b, c, d \in  \tilde{A}_{x_0}$ such that 
 \begin{displaymath}
r = \frac{a-b}{c-d} \in R(\tilde{A}_{x_0}) \not\in  R(B_{y_{0}}).
 \end{displaymath}
By Lemma~\ref{lem:RBcard}, for any subset $Y \subseteq B_{y_{0}}$ with $|Y| \approx |B_{y_{0}}|$, we have
\begin{equation}
\label{eqn:case11}
|B_{y_0}|^{2} \approx |Y|^{2}= |Y - rY| = |aY - bY - cY + dY|.
\end{equation}
By Claim~\ref{coveringapp} and \eqref{eqn:Ax0By0Sizes}, there exists a subset $B^{'} \subseteq B_{y_{0}}$, with $|B^{'}| \approx |B_{y_{0}}|$, such that $dB^{'}$ is contained in 
 \begin{displaymath}
O\bigg(\frac{|A(A+1)|^4}{LN^3}\bigg)
 \end{displaymath}
 translates of $-x_0 A$ and $aB^{'}, bB^{'}, cB^{'}$ are contained in at most the same number of translates of $x_0A$. Thus, setting $Y = B^{'}$, by \eqref{eqn:case11}, we have
 \begin{displaymath}
\bigg(\frac{LN}{|A|} \bigg)^{2} \ll |A-A-A-A|\bigg(\frac{|A(A+1)|^4}{LN^3}\bigg)^{4}.
 \end{displaymath}
Then, by \eqref{iteratednound}, we get
 \begin{displaymath}
M^{6}N^{2}|A|^{21} \ll |A(A+1)|^{40}.
 \end{displaymath}
By \eqref{eqn:SEMECS} and \eqref{eqn:NLB}, we conclude the inequality
 \begin{displaymath}
|A(A + 1)|^{48} \gg (\log|A|)^{-8}|A|^{49}.
 \end{displaymath}
\subsection*{ Case 1.2:\nopunct}
There exist elements $a, b, c, d \in B_{y_{0}}$ such that 
 \begin{displaymath}
r = \frac{a-b}{c-d} \in R(B_{y_{0}}) \not\in R(\tilde{A}_{x_0}).
 \end{displaymath}
Then for any subset $Y \subseteq \tilde{A}_{x_0} $ with $|Y| \approx |\tilde{A}_{x_0}|$, by Lemma~\ref{lem:RBcard}, we have
\begin{equation}
\label{eqn:case12}
|\tilde{A}_{x_0}|^{2} \approx |Y|^{2} = |Y - rY| = |aY - bY - cY + dY|.
\end{equation}
By Claim~\ref{coveringapp}, there exists a subset $A^{'} \subset \tilde{A}_{x_0}$, with $|A^{'}| \approx |\tilde{A}_{x_0}|$, such that the sets $aA^{'}$, $bA^{'}$ and $cA^{'}$ are each fully contained in $O(\Gamma)$ translates of $y_0 A$ and $dA^{'}$ can be covered by $O(\Gamma)$ translates of $-y_0 A$. Thus, setting $Y = A^{'}$, by \eqref{eqn:case12}, we have
 \begin{displaymath}
\bigg(\frac{LM}{|A|^{3}} \bigg)^{2} \ll |A-A-A-A|\bigg(\frac{|A|^{2}|A(A+1)|^4}{M^2}\bigg)^{4}.
 \end{displaymath}
Applying \eqref{iteratednound} yields
 \begin{displaymath}
M^{10}L^{2}|A|^{9} \ll |A(A+1)|^{40}.
 \end{displaymath}
Hence, by \eqref{eqn:SEMECS} and \eqref{eqn:NLB}, we get
 \begin{displaymath}
|A(A + 1)|^{52} \gg (\log|A|)^{-12}|A|^{53}.
 \end{displaymath}
\subsection*{ Case 2:\nopunct} $1 + R(\tilde{A}_{x_0}) \nsubseteq R(\tilde{A}_{x_0})$.
There exist elements $a, b, c, d \in \tilde{A}_{x_0}$ such that
 \begin{displaymath}
r = 1 + \frac{a - b}{c - d} \not\in R(\tilde{A}_{x_0}) = R(B_{y_{0}}).
 \end{displaymath}
Let $Y_1 \subseteq B_{y_0}$ and $Y_2 \subseteq S_a$ be any sets with $|Y_1| \approx |B_{y_0}|$ and $|Y_2| \approx |S_a|$.
By Lemma~\ref{lem:PlRuRefined}, with $X = (c-d)Y_1$, there exists a subset $Y_1^{'} \subseteq Y_1$, with $|Y_1^{'}| \approx |Y_1|$, such that 
\begin{align}
\label{eqn:case2PR}
|Y_1^{'} - rY_2| &= |(c-d)Y_1^{'} - (c-d)Y_2 - (a-b)Y_2|\\ \nonumber &\ll \frac{|Y_1 - Y_2|}{|Y_1|}|(c-d)Y_1- (a-b)Y_2|.
\end{align}
Recall that $Y_1^{'} \subseteq B_{y_0}$ and $Y_2 \subseteq S_a \subseteq B_{y_0}$. Then Lemma \ref{lem:RBcard} gives
 \begin{displaymath}
|Y_1^{'}| |Y_2| = |Y_1^{'} - rY_2|.
 \end{displaymath}
Thus, by \eqref{eqn:case2PR} we have
\begin{equation}
\label{eqn:case2CD}
|Y_1^{'}||Y_1||Y_2| \ll |Y_1- Y_2| |cY_1 -dY_1 -aY_2 +bY_2|.
\end{equation}
Since $Y_1, Y_2\subseteq B_{y_0} \subseteq A+1$, we have
 \begin{displaymath}
|Y_1 -Y_2| \leq |A -A|.
 \end{displaymath}
Recall that $|Y_1^{'}| \approx |Y_1| \approx |B_{y_0}|$ and $|Y_2| \approx |S_a|$. Then by \eqref{eqn:Ax0By0Sizes}, \eqref{eqn:SzSize} and noting that $aY_2 \subseteq x_0 A$, we have
\begin{equation}
\label{eqn:case2reduced}
\bigg(\frac{LN}{|A|} \bigg)^{2}\bigg(\frac{LMN}{|A|^{4}} \bigg) \ll |A- A| |cY_1 -dY_1 -x_0 A +bY_2|.
\end{equation}

Now, by Claim~\ref{coveringapp}, there exist positively proportioned subsets $B_{y_0}^{'} \subseteq B_{y_0}$ and $S_a^{'} \subseteq S_a$ such that $cB_{y_0}^{'}$ and $dB_{y_0}^{'}$ can be covered by 
 \begin{displaymath}
O\bigg(\frac{|A(A+1)|^4}{LN^3}\bigg)
 \end{displaymath}
 translates of $x_0 A$ and $bS_a^{'}$ can be covered by 
 \begin{displaymath}
O\bigg(\frac{|A|^{3}|A(A+1)|^4}{LMN^3}\bigg)
 \end{displaymath}
 translates of $-x_0 A$. Thus, setting $Y_1 = B_{y_0}^{'}$ and $Y_2 = S_a^{'}$, by \eqref{eqn:case2reduced} it follows that
 \begin{displaymath}
\bigg(\frac{LN}{|A|} \bigg)^{2}\bigg(\frac{LMN}{|A|^{4}} \bigg) \ll |A-A||A-A-A-A|\bigg(\frac{|A|^{3}|A(A+1)|^4}{LMN^3}\bigg)\bigg(\frac{|A(A+1)|^4}{LN^3}\bigg)^2.
 \end{displaymath}
Using \eqref{differencebound} and \eqref{iteratednound}, this is further reduced to
 \begin{displaymath}
M^{8}|A|^{21} \ll |A(A+1)|^{44}.
 \end{displaymath}
Thus, by \eqref{eqn:SEMECS}, we get
 \begin{displaymath}
|A(A + 1)|^{52} \gg (\log|A|)^{-8}|A|^{53}.
 \end{displaymath}

\subsection*{Case 3:\nopunct} $x_0^{-1} \tilde{A}_{x_0}\cdot R(\tilde{A}_{x_0}) \nsubseteq R(\tilde{A}_{x_0})$.
There exist elements $a$, $b$, $c$, $d$, $e \in \tilde{A}_{x_0}$ such that
 \begin{displaymath}
r = \frac{a}{x_0}\frac{b - c}{d - e} \not\in R(\tilde{A}_{x_0} ) = R(B_{y_{0}}).
 \end{displaymath}
Given any set $Y_1 \subseteq B_{y_{0}}$, recalling that $S_a \subseteq B_{y_{0}}$, it follows from Lemma~\ref{lem:RBcard} that
 \begin{displaymath}
|Y_1||S_a| = |Y_1 - rS_a|.
 \end{displaymath}
For an arbitrary set $Y_2$, we apply Lemma~\ref{lem:PlRu}, with $X = \frac{b-c}{d-e}Y_2$, to get
\begin{align*}
|Y_2||Y_1||S_a| &= |Y_2||Y_1 - rS_a| \\ &\leq \bigg|Y_1 + \frac{b -c}{d -e} Y_2 \bigg| \bigg|Y_2 - \frac{a}{x_0}S_a \bigg| \\ & \leq |dY_1 -eY_1 +bY_2 -cY_2| |Y_2 - A |.
\end{align*}

By Claim~\ref{coveringapp}, we can identify sets $C_1 \subseteq S_d$ and $C_2 \subseteq P_{c/x_0}$ with $|C_1| \approx |S_d|$ and $|C_2| \approx |P_{c/x_0}| \approx N$, such that $eC_1$ is covered by 
 \begin{displaymath}
O\bigg(\frac{|A|^{3}|A(A+1)|^4}{LMN^3}\bigg)
 \end{displaymath}
translates of $x_0 A$ and $bC_2$ is covered by 
 \begin{displaymath}
O\bigg(\frac{|A(A+1)|^4}{|A|N^3}\bigg)
 \end{displaymath}
translates of $-x_0 A$. We set $Y_1 = C_1$ and $Y_2 = C_2$. Then, by \eqref{eqn:ProjectionsSizes}, \eqref{eqn:SzSize} and particularly noting that $dY_1$, $cY_2 \subset x_0 A$ and $Y_2 \subset A+1$, we have
 \begin{displaymath}
N\bigg(\frac{LMN}{|A|^{4}} \bigg)^{2} \ll |A-A||A-A-A-A|\bigg(\frac{|A|^{3}|A(A+1)|^4}{LMN^{3}}\bigg)\bigg(\frac{|A(A+1)|^4}{|A|N^3}\bigg).
 \end{displaymath}
Using \eqref{differencebound} and \eqref{iteratednound} we get
 \begin{displaymath}
M^{6}N^{3}|A|^{20} \ll |A(A+1)|^{40}.
 \end{displaymath}
By \eqref{eqn:SEMECS} and \eqref{eqn:NLB}, we conclude
 \begin{displaymath}
|A(A + 1)|^{49} \gg (\log|A|)^{-9}|A|^{50}.
 \end{displaymath}

\subsection*{Case 4:\nopunct} Suppose that Cases 1-3 do not happen.
Observing that $R(x_0^{-1}\tilde{A}_{x_0}) = R(\tilde{A}_{x_0})$, by Lemma~\ref{lem:QuotientSetSubfield} we deduce that $R(\tilde{A}_{x_0})$ is the field generated by $x_0^{-1}\tilde{A}_{x_0}$. Then according to the assumptions of Theorem~\ref{thm:ShiftExpander}, we consider the following three cases.

\subsection*{Case 4.1:\nopunct} $R(\tilde{A}_{x_0}) = \mathbb{F}_q$ and $|\tilde{A}_{x_0}| > q^{1/2}$. Let $Y$ denote an arbitrary subset of $\tilde{A}_{x_0}$ with $|Y| \approx |\tilde{A}_{x_0}|.$ By Lemma~\ref{lem:BouGlibPivot}, there exists an element $\xi \in\mathbb{F}_q^{*}$ such that $q\ll |Y + \xi Y|.$ Since $R(B_{y_{0}}) = R(\tilde{A}_{x_0}) = \mathbb{F}_q$, there exist elements $a, b, c, d \in B_{y_{0}}$, such that 
 \begin{displaymath}
q \ll |aY - bY + cY - dY|.
 \end{displaymath}
By Claim~\ref{coveringapp}, we can identify a positively proportioned subset $A^{'} \subset \tilde{A}_{x_0}$, such that $aA^{'}$, $bA^{'}$ and $dA^{'}$ can be covered by $O(\Gamma)$ translates of $y_0 A$ and $cA^{'}$ can be covered by $O(\Gamma)$ translates of $-y_0 A$. Thus, setting $Y = A^{'}$, we have
 \begin{displaymath}
q\ll |A -A -A -A|\bigg(\frac{|A|^{2}|A(A+1)|^4}{M^2}\bigg)^{4}.
 \end{displaymath}
By \eqref{iteratednound}, we get 
 \begin{displaymath}
M^{8}|A|^{15}q \ll |A(A+1)|^{40}.
 \end{displaymath}
By \eqref{eqn:SEMECS}, this gives the bound
 \begin{displaymath}
|A(A + 1)|^{48} \gg q(\log|A|)^{-8}|A|^{47}.
 \end{displaymath}
We point out that if $|\tilde{A}_{x_0}| > q^{1/2}$ then one only needs to consider Cases~1.1 and 4.1, since by Lemma~\ref{lem:RBFq} we have $R(\tilde{A}_{x_0}) = \mathbb{F}_q$.

\subsection*{Case 4.2:\nopunct}  Either $R(\tilde{A}_{x_0}) = \mathbb{F}_q$ and $|\tilde{A}_{x_0}| \leq q^{1/2}$ or $R(\tilde{A}_{x_0})$ is a proper subfield and $|A \cap cR(\tilde{A}_{x_0})| \ll |R(\tilde{A}_{x_0})|^{1/2}$ for all $c \in \mathbb{F}_q$. Since $R(\tilde{A}_{x_0})$ is the field generated by $x_0^{-1}\tilde{A}_{x_0}$, we have $\tilde{A}_{x_0} \subseteq x_0R(\tilde{A}_{x_0})$. Hence 
 \begin{displaymath}
|\tilde{A}_{x_0}|^2 =|\tilde{A}_{x_0} \cap x_0R(\tilde{A}_{x_0})|^2 \leq |A \cap x_0R(\tilde{A}_{x_0})|^2 \ll |R(\tilde{A}_{x_0})|.
 \end{displaymath}
Now, recalling that  $R(\tilde{A}_{x_0}) = R(B_{y_{0}})$, by Lemma~\ref{lem:pivotting}, there exist elements $a, b, c, d \in B_{y_{0}}$ such that for any subset $Y \subseteq \tilde{A}_{x_0}$ with $|Y| \approx |\tilde{A}_{x_0}|,$ we have
\begin{equation}
\label{eqn:case5.2}
|Y|^2 \ll |aY - bY + cY - dY|.
\end{equation}
By Claim~\ref{coveringapp}, there exists a subset $A^{'} \subseteq \tilde{A}_{x_0}$, with $|A^{'}| \approx |\tilde{A}_{x_0}|$, such that $cA^{'}$ can be covered by $O(\Gamma)$ translates of $-y_0 A$ and $aA^{'}, bA^{'}, dA^{'}$ can be covered by $O(\Gamma)$ translates of $y_0 A$.  We set $Y = A^{'}$ so that, by \eqref{eqn:case5.2}, we obtain
 \begin{displaymath}
\bigg(\frac{LM}{|A|^{3}}\bigg)^{2} \ll |A-A-A-A|\bigg(\frac{|A|^{2}|A(A+1)|^4}{M^2}\bigg)^{4}.
 \end{displaymath}
Applying \eqref{iteratednound} gives
 \begin{displaymath}
M^{10}L^{2}|A|^{9} \ll |A(A+1)|^{40}.
 \end{displaymath}
Then, by \eqref{eqn:SEMECS} and \eqref{eqn:NLB}, we have
 \begin{displaymath}
|A(A + 1)|^{52} \gg (\log|A|)^{-12}|A|^{53}.
 \end{displaymath}

\subsection*{Case 4.3:\nopunct}  $R(\tilde{A}_{x_0}) $ is a proper subfield and $|A \cap x_0 R(\tilde{A}_{x_0})| \ll |A|^{25/26}$. Recall that $\tilde{A}_{x_0}\subset x_0 R(\tilde{A}_{x_0})$. Then, by \eqref{A1Size} and \eqref{eqn:NLB}, we get 
 \begin{displaymath}
\frac{M^{2}}{|A|^{5}} \ll |\tilde{A}_{x_0}| \ll |A|^{25/26}.
 \end{displaymath}
Using \eqref{eqn:SEMECS}, we recover the bound
 \begin{displaymath}
|A(A + 1)|^{52} \gg (\log|A|)^{-52}|A|^{53}.
 \end{displaymath}

\section{Proof of Corollary~\ref{cor:AdEnergy}}
Let $\alpha \in \mathbb{F}_q^*$ and denote $S = A \cap (A -\alpha)$. Observing that $ S, S + \alpha \subset A,$ we deduce $|S(S + \alpha)| \leq |AA|$. Then, estimate \eqref{eqn:ShiftInter} follows by applying Theorem~\ref{thm:ShiftExpander} to the set $S$. Now, since $S \subset A$, if $A$ satisfies restriction \eqref{eqn:AdEnergyCon}, then $S$ can fail to satisfy restriction \eqref{eqn:ShiftECon} only if $|S| \ll |AA|^{52/53}$, which in fact gives the required estimate. This concludes the proof of estimate \eqref{eqn:ShiftInter}.

Next, noting that 
 \begin{displaymath}
|A \cap (A -\alpha)| = |\{ (a_1, a_2) \in A^2 : a_1 - a_2 = \alpha\}|,
 \end{displaymath}
similarly to \eqref{eqn:r1stM} and \eqref{eqn:r2ndM}, we have the identities
 \begin{displaymath}
 |A|^2 = \sum_{\alpha \in A-A} |A \cap (A -\alpha)| \quad \text{and} \quad E_{+}(A) =  \sum_{\alpha \in A-A} |A \cap (A -\alpha)|^2.
 \end{displaymath}
In particular, it follows that
 \begin{displaymath}
E_{+}(A) \ll  |A|^2\cdot\max_{\alpha \in \mathbb{F}_q^*}|A \cap (A -\alpha)|.
 \end{displaymath}
Thus the required bound on $E_{+}(A)$ follows from \eqref{eqn:ShiftInter}.

\section*{Acknowledgments}
The author would like to thank Igor Shparlinski for helpful conversations.

\end{document}

%% file: ex_shared.tex
\usepackage{amsfonts}
\usepackage{amssymb}
\usepackage{graphicx}
\usepackage{epstopdf}
\usepackage{algorithmic}
\ifpdf
  \DeclareGraphicsExtensions{.eps,.pdf,.png,.jpg}
\else
  \DeclareGraphicsExtensions{.eps}
\fi

\newsiamremark{remark}{Remark}
\newsiamremark{hypothesis}{Hypothesis}
\crefname{hypothesis}{Hypothesis}{Hypotheses}
\newsiamthm{claim}{Claim}

\headers{On growth of the set $A(A+1)$}{Ali Mohammadi}

\title{On growth of the set $A(A+1)$ in arbitrary finite fields}

\author{Ali Mohammadi\thanks{School of Mathematics and Statistics, University of Sydney, NSW 2006, Australia 
  (\email{alim@maths.usyd.edu.au}).}}

\usepackage{amsopn}
